\documentclass[proceedings,submission%
]{dmtcs}

\usepackage[latin1]{inputenc}
\usepackage{subfigure}
\usepackage{amsmath,amssymb,relsize,color,float,rotfloat,afterpage,comment}
\usepackage{youngtab,theorem}

\newcommand{\mysize}[1]{\text{{\small #1}}}

\newcommand{\smalla}{\mysize{0}}
\newcommand{\smallb}{\mysize{1}}
\newcommand{\smallc}{\mysize{2}}
\newcommand{\smalld}{\mysize{3}}
\newcommand{\smalle}{\mysize{4}}
\newcommand{\smallf}{\mysize{5}}
\newcommand{\smallg}{\mysize{6}}
\newcommand{\smallh}{\mysize{7}}
\newcommand{\smalli}{\mysize{8}}

\newtheorem{thm}{Theorem}
\newtheorem{lemma}[thm]{Lemma}

\newtheorem{prop}[thm]{Proposition}

\newcommand{\moinsm}{\mysize{-1}}
\newcommand{\moinsmm}{\mysize{-2}}
\newcommand{\moinsmmm}{\mysize{-3}}

\newcommand{\complexi}{\mathfrak{i}}

\DeclareMathOperator{\UnitOp}{U}

\DeclareMathOperator{\id}{Id}

\newcommand{\fracG}{\frac{G(k+1)^2}{G(2k+1)}}

\newcommand{\Repeat}[2]{\left(\left\{#2\right\}^{#1}\right)}
\newcommand{\Ones}[1]{\Repeat{#1}{1}}

\newcommand{\Unit}[1]{{\UnitOp(#1)}}

\newcommand{\ud}{\text{d}}
\newcommand{\averageLeft}{\left<}
\newcommand{\averageRight}{\right>_\Unit{N}}
\newcommand{\average}[1]{\averageLeft #1 \averageRight}

\newcommand{\schur}{\mathfrak{s}}
\newcommand{\partition}{\vdash}

\author{Paul-Olivier Dehaye}
\title[Moments of derivatives of characteristic polynomials]{A note on moments of derivatives of characteristic polynomials}

\address{ETH Z\"urich, Department of Mathematics, 8092 Z\"urich, Switzerland, \textsf{pdehaye@math.ethz.ch}}
\begin{document}
\maketitle
\begin{abstract}
\paragraph{Abstract.}
We present a simple technique to compute moments of derivatives of unitary characteristic polynomials. The first part of the technique relies on an idea of Bump and Gamburd: it uses orthonormality of Schur functions over unitary groups to compute matrix averages of characteristic polynomials. In order to consider derivatives of those polynomials, we here need the added strength of the Generalized Binomial Theorem of Okounkov and Olshanski. This result is very natural as it provides coefficients for the Taylor expansions of Schur functions, in terms of shifted Schur functions. The answer is finally given as a sum over partitions of functions of the contents.  One can also obtain alternative expressions involving hypergeometric functions of matrix arguments. 

\paragraph{R\'esum\'e.}
Nous introduisons une nouvelle technique, en deux parties, pour calculer les moments de d\'eriv\'ees de polyn\^omes caract\'eristiques. La premi\`ere \'etape repose sur une id\'ee de Bump et Gamburd et utilise l'orthonormalit\'e des fonctions de Schur sur les groupes unitaires pour calculer des moyennes de polyn\^omes caract\'eristiques de matrices al\'eatoires. La deuxi\`eme \'etape, qui est n\'ecessaire pour passer aux d\'eriv\'ees, utilise une g\'en\'eralisation du th\'eor\`eme binomial  due \`a Okounkov et Olshanski. Ce th\'eor\`eme livre les coefficients des s\'eries de Taylor pour les fonctions de Schur sous la forme de ``shifted Schur functions''. La r\'eponse finale est donn\'ee sous forme de somme sur les partitions de fonctions des contenus. Nous obtenons aussi d'autres expressions en terme de fonctions hyperg\'eom\'etriques d'argument matriciel.
\end{abstract}

\section{Introduction}
\label{introduction}
We take for the characteristic polynomial of a $N\times N$ unitary matrix $U$
\begin{eqnarray}
Z_U(\theta) & := & \prod_{j=1}^N \left(1-e^{\complexi(\theta_j-\theta)}\right),
\end{eqnarray}
where the $\theta_j$s are the eigenangles of $U$ and set 
\begin{eqnarray}
V_U(\theta) & := & e^{\complexi N (\theta + \pi)/2} e^{-\complexi
  \sum_{j=1}^N \theta_j/2} Z_U(\theta).
\label{VAppear}
\end{eqnarray}
It is easily checked that for real $\theta$, $V_U(\theta)$ is real and that $|V_U(\theta)|$ equals $|Z_U(\theta)|$. 

For $k$ and $r$ integers, with $0\le r \le 2k$, we will investigate the averages (with respect to Haar measure)
\begin{eqnarray}
\left(\mathcal{M}\right)_N(2k,r)&:=&\average{\left|Z_U(0)\right|^{2k} \left( \frac{Z_U'(0)}{Z_U(0)} \right)^{r}},\\
\left|\mathcal{V}\right|_N(2k,r)&:=&\average{\left|V_U(0)\right|^{2k} \left|
    \frac{V_U'(0)}{V_U(0)} \right|^{r}}.
\end{eqnarray}
This is notation we already used in \cite{DehayeJoint} (where a $\left|\mathcal{M}\right|$ was also present but is not needed here), and is only notation in the LHS: the $\left(\mathcal{M}\right)$ and  $\left|\mathcal{V}\right|$ are thus meant each as \emph{one} symbol and are supposed to mnemotechnically remind the reader of what is in the RHS. We 
immediately state the following easy lemma.
\begin{lemma} 
\label{firstLemma}
For $k\ge h$ non-negative integers, we have the relation
\begin{eqnarray}
\left|\mathcal{V}\right|_N(2k,2h)&=& \sum_{i=0}^{2h} \binom{2h}{i} \left(\mathcal{M}\right)_N(2k,i) \left(\frac{\mathfrak{i}N}{2}\right)^{2h-i}.
\end{eqnarray}
\end{lemma}
\begin{proof}
This is available (in the same notation) in \cite{DehayeJoint} and a consequence of Equation~(\ref{VAppear}), which leads to the polynomial relations between $\frac{Z_U'(0)}{Z_U(0)}$, $\frac{V_U'(0)}{V_U(0)}$ and their norms. These relations give 
\begin{eqnarray}
\left|\mathcal{V}\right|_N(2k,2h)  &=& \sum_{j=0}^{h} \binom{h}{j}
\left(\frac{-N^2}{4}\right)^{h-j} \sum_{l=0}^j (\complexi N)^{j-l}
\binom{j}{l} (\mathcal{M})_N(2k,j+l),
\end{eqnarray}
which is easily deduced from \cite{DehayeJoint}. 
\end{proof}

We are actually more concerned with the renormalizations
\begin{eqnarray}
\left(\mathcal{M}\right)(2k,r) &=&\lim_{N \rightarrow \infty} \frac{\left(\mathcal{M}\right)_N(2k,r)}{N^{k^2+r}}\\
\left|\mathcal{V}\right|(2k,r) &=&\lim_{N \rightarrow \infty} \frac{\left|\mathcal{V}\right|_N(2k,r)}{N^{k^2+r}}.
\end{eqnarray}
Theorem~\ref{sumPartitionsThm} will show that these normalizations are appropriate.

The random matrix theory problem of evaluating $\left(\mathcal{M}\right)(2k,r)$ and $\left|\mathcal{V}\right|(2k,r)$ has applications in number theory (see \cite{DehayeJoint} for a more detailed exposition of these ideas). Indeed, these values are related to the factor $g(k,h)$ in the formula
\begin{eqnarray}
\lim_{T \rightarrow \infty}\frac{1}{T}\frac{1}{(\log \frac{T}{2\pi})^{k^2+2h}} \int_0^T \left|\zeta\left(\frac{1}{2} +\complexi t\right)\right|^{2k-2h} \left|\zeta'\left(\frac{1}{2} +\complexi t\right)\right|^{2h} \ud t &=& a(k) g(k,h),
\end{eqnarray}
where $a(k)$ is a (known) factor defined as a product over primes. This, along with a discrete moment version due to Hughes, is the principal underlying motivation for the all the random matrix theory analysis that occurs in \cite{HughesFirst,HughesThesis,HughesJoint,Mezzadri2003,CRS,ForresterWitte}.

Another application is tied to the work of Hall \cite{Hall2002,HallWirtinger,Hall2004,Hall}, where results on the objects studied here can be used to hint towards optimizations of rigorous arguments in number theory, and serve as (conjectural) inputs on theorems there. The number theory statements concern average spacings between zeroes of the Riemann zeta function. The works of Steuding \cite{Steuding} and Saker \cite{Saker} follow similar approaches.

In \cite{DehayeJoint}, the author investigated for fixed $r$ ratios of those quantities and established they were a rational function. We reprove this result here, but with a much simpler method leading to a much simpler result. In particular, rationality of the RHS is transparent from the following statement (definitions are given in Section~\ref{definitions}), since the RHS sums are finite sums over partitions of $r$:

\begin{thm}
\label{sumPartitionsThm}
For $0 \le r\le 2k$, with $r,k \in \naturals$,
\begin{eqnarray}
\frac{(\mathcal{M})_N(2k,r)}{(\mathcal{M})_N(2k,0)}\complexi^r &=&  
  \sum_{\mu\partition r}\frac{r!}{h_\mu^2} \frac{(N\uparrow \mu) ((-k) \uparrow \mu)  }{(-2k) \uparrow \mu},
\label{finiteNformula}
\\
\label{limitFormula}
\frac{(\mathcal{M})(2k,r)}{(\mathcal{M})(2k,0)}\complexi^r &=&  
  \sum_{\mu\partition r}\frac{r!}{h_\mu^2} \frac{ k \uparrow \mu  }{(2k) \uparrow \mu},
\end{eqnarray}
while the denominators on the left are known (\cite{BumpGamburd},\cite{KS}):
\begin{eqnarray}
(\mathcal{M})_N(2k,0) &=& \schur_{\left<N^k\right>}\Ones{2k}=\frac{G(N+2k+1) G(N+1) }{ G(N+k+1)^2} \fracG
\end{eqnarray}
and
\begin{eqnarray}
 (\mathcal{M})(2k,0) &=& \lim_{N\rightarrow \infty} \frac{ \schur_{\left<N^k\right>}\Ones{2k}}{N^{k^2}}=\fracG,
\end{eqnarray}
where $G(\cdot)$ is the Barnes $G$-function.
\end{thm}

We now briefly discuss the technique used to obtain this theorem. The first idea will be similar to an idea of Bump and Gamburd \cite{BumpGamburd} of using orthonormality of Schur functions to efficiently compute matrix averages. The new idea here is to combine this with the Generalized Binomial Theorem (\ref{binomial}) of Okounkov and Olshanski \cite{OkOl1997} in order to obtain information about the moments of derivatives instead of moments of polynomials directly. 

This paper is structured as follows. 
We give in Section~\ref{definitions} the basic definitions needed. In Section~\ref{shiftedSec}, we explain the Generalized Binomial Theorem. We prove Theorem~\ref{sumPartitionsThm} in Section~\ref{proofSec}. We use this result to deduce in Section~\ref{propertiesSec} further properties of the rational functions obtained. We present in Section~\ref{hyperSec} an alternative interpretation of these results in terms of hypergeometric functions of a matrix argument. Finally, we announce briefly in Section~\ref{furtherSec} further results.

\section{Definitions}
\label{definitions}
Since Theorem~\ref{sumPartitionsThm} presents its result as a sum over partitions, we first need to define some classical objects associated to them. We follow conventions of \cite{StanleyBook} throughout.

Partitions are weakly decreasing sequences $\lambda_1 \ge \lambda_2 \ge \cdots \lambda_{l(\lambda)}$ of positive integers, its \emph{parts}. The integer $l=l(\lambda)$ is called the \emph{length} of the partition $\lambda$. We call the sum $\sum_i \lambda_i$ of its parts the \emph{size} $|\lambda|$ of the partition $\lambda$. We sometimes say that $\lambda$ partitions $|\lambda|$, which is written $\lambda \vdash |\lambda|$. If the partition has $k$ parts of equal size $N$, we simplify notation to $\left<N^k\right>$. 

We always prefer to think of partitions graphically. To each partition we associate a \emph{Ferrers diagram}, \textit{i.e.}~ the Young diagram of the partition presented in the English convention. We only give one example (Figure \ref{DiagramFig}) as it should be clear from it how the diagram is constructed: the parts $\lambda_i$ indicate how many standard boxes to consider on each row. 
\begin{figure}
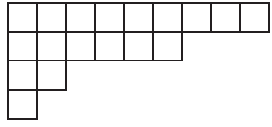

\centering
\yng(9,6,2,1)
\caption{The Ferrers diagram of partition $(9,6,2,1)$.}
\label{DiagramFig}
\end{figure}

There exists an involution acting on partitions, which we denote by $\lambda^t$. Its action on diagrams amounts to a reflection along the main diagonal.

Partitions can be indexed in many different ways. Indeed, we have already seen that finite sequences of (decreasing) part sizes can be used as an index set. Shifted part lengths are essential for the work of Okounkov and Olshanski underlying Section \ref{shiftedSec}, but we do not need to define that system of coordinates. 

\begin{figure}
\begin{displaymath}
\young(\smalla \smallb \smallc \smalld \smalle \smallf \smallg \smallh \smalli,\moinsm\smalla\smallb\smallc\smalld\smalle,\moinsmm\moinsm,\moinsmmm)
\quad\quad\quad
\young({{12}}{{10}}8765321,864321,31,1)
\end{displaymath}
\caption{The contents and hook lengths of the partition $(9,6,2,1)$. Its hook number is thus 4\! 180\! 377\! 600.}
\label{contentHookFig}
\end{figure}

Define the \emph{content} of a box $\square$ located at position $(i,j)$ in a partition $\lambda$ as $ c(\square) = j-i $ (see Figure~\ref{contentHookFig}). We use this to define the symbol
\begin{eqnarray}
k \uparrow \mu &:=& \prod_{\square \in \mu} (k+c(\square)).
\end{eqnarray}
This definition leads to $k \uparrow (n) = k (k+1)\cdots(k+n-1)$ and $k \uparrow (1^n) = k (k-1) \cdots (k-n+1)$. We sometimes abbreviate the first $k\uparrow n$ and the second $k \downarrow n$. This is clearly a generalization of the Pochhammer symbol. Indeed, we even adopt the convention that $N \uparrow (-k) = 1/((N+1)\uparrow k)$, which guarantees 
\begin{eqnarray}
\label{consistency}
(N+a-1)\downarrow (a+b) &= &(N\uparrow a)\cdot ((N-1)\downarrow b).
\end{eqnarray} 

To get back to the generalization to partitions, we have an immediate relation under conjugation:
\begin{eqnarray}
\label{conjugation}
k\uparrow \mu^t &=& (-1)^{|\mu|} ((-k) \uparrow \mu).
\end{eqnarray}

\begin{comment}
Also, for large enough $R$, we can reorganize the product in the initial definition along rows and obtain
\begin{eqnarray}
k \uparrow \mu &=& \prod_{1 \le i \le R} \frac{(\mu_i+n-i)!}{(n-i)!}.
\end{eqnarray}
This allows for extensions to $\mu \in \mathbb{C}^\infty$ (or at least subsets where this converges) using the $\Gamma$-function.
\end{comment}

\begin{comment}
We also have 
\begin{eqnarray}
h_\mu = \prod_{\alpha \in \mu} h_\alpha = \prod_{1 \le i < j \le R} \frac{(\mu_i -\mu_j+j-i)!}{(j-i)!},
\end{eqnarray}
which similarly allows for analytic extension to ``continuous partitions'' (\textit{i.e.}~the parts are continuous), or rather cones.
Both functions $X \uparrow \mu $ and $H(\mu)$  are symmetric in the variables $\mu_i -i$. 

It is also useful to observe that 
\begin{eqnarray}
\dim \mu &=& \frac{|\mu|!}{H(\mu)}
\label{dimension}
\end{eqnarray}
and
\begin{eqnarray}
\label{evaluationSchur}
\schur_\mu\left(\left[1^N\right]\right) &=& \frac{N \uparrow \mu}{H(\mu)}.
\end{eqnarray}
\end{comment}

Given a box $\square \in \lambda$, define its hook (set)
\begin{eqnarray}
h_\square = h_{(i,j)} &:=& \left\{ (i,j')\in \lambda:j'\ge j \right\} \cup \left\{ (i',j) \in \lambda:i'\ge i \right\}.
\end{eqnarray}
Remark that the box $\square$ itself is in its hook. Define the \emph{hook length} $|h_{(i,j)}|$ as the cardinality of the hook (see Figure~\ref{contentHookFig}) and call their product the \emph{hook number} of a partition $\lambda$:
\begin{eqnarray}
h_\lambda &:=& \prod_{\square \in \lambda} h_\square.
\end{eqnarray}

Hook numbers are of importance thanks to the hook length formula of Frame, Robinson and Thrall \cite{FRT}. This counts the number $f_\lambda$ of standard tableaux of shape $\lambda$, which is also the dimension $\dim \chi^\lambda = \chi^\lambda(1)$ of the character  associated to the partition $\lambda$ for the symmetric group  $\mathcal{S}_{|\lambda|}$ (see \cite{Sagan,StanleyBook}):
\begin{eqnarray}
\label{dimensionsFormula}
f_\lambda&:=& \dim \chi_{\mathcal{S}_{|\lambda|}}^\lambda= \frac{|\lambda|!}{h_\lambda}.
\end{eqnarray}

The last classical combinatorial object we need is the \emph{Schur functions}. To each partition $\lambda$ we associate an element $\schur_\lambda$ of degree $|\lambda|$ of the ring $\Lambda_{\mathbb{Q}}[X]$ of polynomials symmetric in a countable set of variables $X=\left\{ x_i \right\}$. We refer the reader to \cite{BumpLieGroups} for definitions, and only state a few properties.

Let $r\in \naturals$, and take a partition $\lambda$ of size $r$. One can define a map from $\chi^\lambda_N$ from $\Unit{N}$ to $\mathbb{C}$ in the following way:
\begin{eqnarray}
\label{eigenvalues}
\chi^\lambda_N(g):= \schur_\lambda(g) := \schur_\lambda(e^{\mathfrak{i} \theta_1},\cdots,e^{\mathfrak{i} \theta_N},0,0,0,\cdots),
\end{eqnarray}
with the $e^{\mathfrak{i} \theta_j}$ the eigenvalues of $g \in \Unit{N}$. When $l(\lambda)>N$, $\chi^\lambda_N(g)\equiv 0$, but once $N \ge l(\lambda)$, the $\chi^\lambda_N(g)$ become (different) irreducible characters of $\Unit{N}$. We thus have the formula
\begin{eqnarray}
\average{\chi^\lambda_N,\chi^\mu_N} =
\average{\schur_\lambda(\cdot),\schur_\mu(\cdot)}= 
\left\{
\begin{array}{cl}
1 & \text{ if $\lambda = \mu$ and $N\ge |\lambda|$},\\
0 & \text{ otherwise.}
\end{array}
\right.
\end{eqnarray}
The last formula we need concerns the evaluation of Schur polynomials, at repeated values of the arguments.  Denote by $\{a\}^R$ the multiset consisting of the union of $R$ copies of $a$ and countably many copies of $0$. The hook-content formula (see \cite{StanleyBook,BumpLieGroups}), a consequence of the Weyl Dimension Formula, states then that
\begin{eqnarray}
\label{hookcontentFormula}
\schur_\lambda \Repeat{k}{1} &=&\frac{k \uparrow \lambda}{h_\lambda}.
\end{eqnarray}
\section{Shifted Schur Functions and Generalized Binomial Theorem}
\label{shiftedSec}
The Generalized Binomial Theorem as formulated in \cite[Theorem 5.1]{OkOl1997} can be interpreted as a Taylor expansion of the character $\chi^\lambda_n=\schur_\lambda(\cdot)$ of $\Unit{n}$ around the identity $\id_{n \times n}$.  It says explicitly that
\begin{eqnarray}
\label{binomial}
\frac{\schur_\lambda(1+x_1,\cdots,1+x_n)}{\schur_\lambda\Ones{n}} & = &
\sum_{\substack{\mu\\l(\mu)\le n }} \frac{\schur_\mu^*(\lambda_1,\cdots,\lambda_n)\schur_\mu(x_1,\cdots,x_n)}{n \uparrow \mu},
\end{eqnarray}
where the $\schur_\mu^*$ are shifted Schur functions. Those were introduced by Okounkov and Olshanski in \cite{OkOl1997} and need not be defined here as we only need their values on a very limited set of arguments.
This is given by the following lemma, which is very elegant and seems to be new. Note the (almost-)symmetry between $k$ and $N$.
\begin{lemma}
\label{elegant}
\begin{eqnarray}
\schur_\mu^*\Repeat{k}{N} &=& h_\mu \cdot  \schur_{\mu^t}\Ones{N} \cdot \schur_\mu\Ones{k} \\
&=&(-1)^{|\mu|} \frac{((-N) \uparrow \mu)(k\uparrow \mu)}{h_\mu}
\end{eqnarray}
\end{lemma}
\begin{proof}
We will need two equations from \cite{OkOl1997}. Equation~(11.28) tells us that
\begin{eqnarray}
\schur_\mu^*(x_1,\cdots,x_n) &=& 
\det \left[\mathfrak{h}^*_{\mu_i-i+j}(x_1+j-1,\cdots,x_n+j-1) \right]_{i,j=1}^{R,R}
\end{eqnarray}
for large enough $R$ (it is then stable in $R$). Equation~(11.22) from \cite{OkOl1997} deals precisely with those $\mathfrak{h}^*$:
\begin{eqnarray}
\mathfrak{h}^*_r\Repeat{k}{N} &=& (N\downarrow r)\cdot\mathfrak{h}_r\Ones{k}.
\end{eqnarray}
Combining these equations, we obtain
\begin{eqnarray}
\schur_\mu^* \Repeat{k}{N} &=& 
\det \left[((N+j-1)\downarrow (\mu_i-i+j))  \cdot \mathfrak{h}^*_{\mu_i-i+j}\Ones{k} \right]_{i,j=1}^{R,R}\\
&=&\det \left[(N\uparrow j) ((N-1)\downarrow (\mu_i-i))  \cdot \mathfrak{h}^*_{\mu_i-i+j}\Ones{k} \right]_{i,j=1}^{R,R}\\
&=& \left(\prod_{i=1}^R ((N-1)\downarrow (\mu_i-i)) \right) \left(\prod_{j=1}^R (N\uparrow j) \right) \det \left[\mathfrak{h}^*_{\mu_i-i+j}\Ones{k}\right]_{i,j=1}^{R,R}\\
&=& \left(\prod_{i=1}^R ((N+i-1)\downarrow \mu_i) \right) \schur_\mu\Ones{k}\\
&=& \left(\prod_{\square \in \mu} N-c(\square) \right)\schur_\mu\Ones{k}\\
&=& h_\mu \schur_{\mu^t}\Ones{N}\schur_\mu\Ones{k}.
\end{eqnarray}
The fourth line follows from Equation~(\ref{consistency}), the fifth from reorganizing a product over rows into a product over boxes, and the sixth from Equation~(\ref{hookcontentFormula}).
\end{proof}

We are now ready to launch into the proof of Theorem~\ref{sumPartitionsThm}. 

\section{Proof of Theorem~\ref{sumPartitionsThm}}
\label{proofSec}
The method of proof will be very similar to the technique presented in \cite{BumpGamburd}. In particular, both the Cauchy identity \cite{BumpLieGroups}
\begin{eqnarray}
\label{Cauchy}
\prod_{i,j}{1+x_iy_j} &=& \sum_\lambda \schur_{\lambda^t}(x_i)
\schur_\lambda(y_j),
\end{eqnarray}
where $x_i$ and $y_j$ are finite sets of variables, and the asymptotic orthonormality of the Schur functions will again play a crucial role. However, the power of their technique is now supplemented by the Generalized Binomial Theorem, which will provide  for a dramatic simplification of the arguments and results in \cite{DehayeJoint}.

\begin{proof}[of Theorem~\ref{sumPartitionsThm}]
We have
\begin{eqnarray}
\overline{Z_U(0)}  &=&  \prod_{j=1}^N \left(1- e^{-\complexi \theta_j}\right)\\
&=&\prod_{j=1}^N -e^{-\complexi \theta_j} \left(1- e^{\complexi \theta_j}\right)\\
&=& (-1)^N \overline{\det U} Z_U(0)
\end{eqnarray}
and thus
\begin{eqnarray}
\overline{Z_U(0)}^k & = &  (-1)^{kN} \overline{\det U}^k Z_U(0)^k\\
&=& (-1)^{kN} \overline{\schur_{\left<k^N\right>}(U)} Z_U(0)^k.
\end{eqnarray}
We use the Cauchy Identity from Equation~(\ref{Cauchy}) to obtain
\begin{eqnarray}
Z_U(a_1)\cdots Z_U(a_r) & = &\sum_\lambda \schur_{\lambda^t}(U) \schur_\lambda\left(-e^{-\complexi a_1},\cdots,-e^{-\complexi a_r}\right)\\
& = &\sum_\lambda (-1)^{|\lambda|} \schur_{\lambda^t}(U) \schur_\lambda\left(e^{-\complexi a_1},\cdots,e^{-\complexi a_r}\right).
\end{eqnarray}
To the first order in small $a$, we have $ e^{-\complexi a} \approx 1- \complexi a$, so
\begin{eqnarray}
Z_U'(0)^r & = & \sum_\lambda (-1)^{|\lambda|} \schur_{\lambda^t}(U) \,\, \partial_{1}\cdots \partial_{r} \bigl. \schur_\lambda\left(1-\complexi a_1,\cdots,1-\complexi a_r\right)\bigr|_{a_1=\cdots=a_r=0},
\end{eqnarray}
where $\partial_i := \partial_{a_i}$.

Putting everything together, we obtain
\begin{multline}
\left|Z_U(0)\right|^{2k} \left( \frac{Z_U'(0)}{Z_U(0)} \right)^{r}=
(-1)^{(kN)} \overline{\schur_{\left<k^N\right>}(U)} \cdot\hfill \\ \sum_\lambda (-1)^{|\lambda|} \schur_{\lambda^t}(U) \,\, \partial_{1}\cdots \partial_{r} \bigl. \schur_\lambda\left((2k-r) \times \left\{1\right\} \cup \left\{1-\complexi a_1,\cdots,1-\complexi a_r\right\}\right)\bigr|_{a_1=\cdots=a_r=0}.
\end{multline}
Just as in the original proof of Bump and Gamburd, orthogonality of the Schur polynomials kills all terms in the sum but one (where $\lambda^t$ equals $  \left<k^N\right>$) under averaging over $\Unit{N}$. Hence this simplifies to
\begin{multline}
\average{\left|Z_U(0)\right|^{2k} \left( \frac{Z_U'(0)}{Z_U(0)} \right)^{r}}
=\\
\partial_{1}\cdots \partial_{r} \bigl. \schur_{\left<N^k\right>}\left((2k-r) \times \left\{1\right\} \cup \left\{1-\complexi a_1,\cdots,1-\complexi a_r\right\}\right)\bigr|_{a_1=\cdots=a_r=0}.
\end{multline}
This is the perfect opportunity to apply the Generalized Binomial Theorem.
 We wish to set $n=2k$, and 
\begin{eqnarray}
x_i &=& \left\{ \begin{array}{ccl} -\complexi a_i &\text{ for }& 1\le i \le r\\ 0 &\text{ for } &r+1 \le i \le 2k\end{array}\right.
\end{eqnarray}
in Equation~(\ref{binomial}) to get
\begin{multline}
\average{\left|Z_U(0)\right|^{2k} \left( \frac{Z_U'(0)}{Z_U(0)} \right)^{r}}= \schur_{\left<N^k\right>}\Ones{2k}\times\\
 (-\complexi)^r \sum_{\mu\partition r} \frac{\schur_\mu^*\Repeat{k}{N}\,\,\partial_1\cdots\partial_r\bigl.\schur_\mu(a_1,\cdots,a_r)\bigr|_{a_1=\cdots=a_r=0}}{(2k) \uparrow \mu}.
\end{multline}
The (additional) restriction on $\mu$ is obtained because of the derivatives: since the Schur functions are evaluated at $a_i=0$, and $\schur_\mu$ is of total degree $|\mu|$ in the $a_i$, we must have $|\mu|= r$ for something to survive $\partial_1 \cdots \partial_r$. 

When $\mu \partition r$, we have\footnote{Observe that higher derivatives require knowledge of values of symmetric group characters at other group elements than the identity.} 
\begin{eqnarray} \partial_1\cdots\partial_r\bigl.\schur_\mu(a_1,\cdots,a_r)\bigr|_{a_1=\cdots=a_r=0} &=&  \left< \schur_\mu, \mathfrak{p}_{\left<1^r\right>}\right>
= \dim \chi^\mu_{\mathcal{S}_{|\mu|}}
= \frac{|\mu|}{h_\mu}
\end{eqnarray}
since derivation and multiplication by power sums are adjoint.

Combined with Lemma~(\ref{elegant}), Equation~(\ref{conjugation}) and Equation~(\ref{hookcontentFormula}), this gives Equation~(\ref{finiteNformula}), which then quickly implies (\ref{limitFormula}).
\end{proof}

\section{Properties of the Rational Functions}
\label{propertiesSec}
In the RHS of formulas (\ref{finiteNformula}) and (\ref{limitFormula}), we have a sum of rational multiples of ratios of polynomials in $k$ (and $N$), hence rational functions of $k$. We will now explain some of the properties of these functions, which are easily deduced from Equation~ (\ref{limitFormula}) and the limiting version of Lemma~\ref{firstLemma}.
\begin{prop}
For a fixed $r,h \in \naturals$, there exists sequences of even polynomials $X_r, Y_r$ and $\tilde{X}_{2h}$ such that for all $k\ge r,2h$,
\begin{eqnarray}
\frac{(\mathcal{M})(2k,r)}{(\mathcal{M})(2k,0)} &=&  \left( -\frac{\complexi}{2}\right)^r\frac{X_r(2k)}{Y_r(2k)},\\
\frac{\left|\mathcal{V}\right|(2k,2h)}{\left|\mathcal{V}\right|(2k,0)} &=&  \frac{\tilde{X}_{2h}(2k)}{Y_{2h}(2k)},
\end{eqnarray}
and such that $X_r$ and $Y_r$ are of the same degree, monic, and with integer coefficients. We also have $\deg \tilde{X}_{2h} \le \deg Y_{2h}$ and that $\tilde{X}_{2h}$ has integer coefficients.
\end{prop}
This Proposition leads or relies on a few easy facts.
\paragraph{Remarks.}
\begin{itemize}
\item All polynomials are even due to Equation~(\ref{conjugation}).
\item We emphasize that we are not quite taking the simplest possible form of $X_r(2k)$, $\tilde{X}_r(2k)$ and $Y_r(2k)$ here. For some (presumably small) $r$s, some spurious cancellations \emph{will} occur. However, an explicit expression, given in \cite{DehayeJoint}, can be obtained for the $Y_r$ such that all the previous statements are true. Let us just say that we are taking for $Y_r$ the common denominator not of every term in the sum over partitions $\lambda$, but of every (1-or-2-terms-)subsum over orbits of the involution on partitions, \emph{after} all the simplifications of the type  $\frac{k+i}{2k+2i}=\frac{1}{2}$ that occur in $\frac{k\uparrow \lambda}{(2k)\uparrow \lambda}$.
\item With this convention, the zeroes of $Y_r$ are exactly at the odd integers between $1-r$ and $r-1$.
\item Since the squares of dimension of characters of a finite group $G$ sum to the order of $G$, or due to the existence of the Robinson-Schensted-Knuth correspondence between permutations of $n$ and pairs of Young tableaux of the same shape partitioning $n$, we have starting from Equation~(\ref{dimensionsFormula}) the identity
\begin{eqnarray}
\sum_{\lambda \vdash n} \frac{n!}{h_\lambda^2} = 1 ,
\end{eqnarray}
which defines the \emph{Plancherel measure} on partitions of $n$. The polynomial $X_r$ is monic thanks to this last identity. 
\item The sum appearing in Equation~(\ref{limitFormula}) is a special case of a problem studied by Jonathan Novak in Equation~(9.21) of his thesis \cite{Novak}.
\end{itemize}

\section{Hypergeometric Functions of Matrix Arguments}
\label{hyperSec}
For completeness, we now discuss an alternative way to approach the expression in (\ref{finiteNformula}). Hypergeometric functions of a $N \times N$ matrix argument $M$ are generalizations of hypergeometric functions of a complex variable. Originally, this extension is defined on multisets of complex numbers, which can then be seen as eigenvalues of a matrix using the trick of Equation~(\ref{eigenvalues}).  There is extensive literature on those functions, most of it tied to multivariate statistical analysis. For a recent, accessible presentation, consider \cite{importanceSelberg}.

For $M$ a $N\times N$ matrix, we define \cite{RamanujanMaster,GrRi1991} a function of $M$ as follows:
\begin{eqnarray}
\label{pFq}
_pF_q(a_i,b_j;M) &=& \sum_\lambda \frac{\prod_{i=1}^p a_i\uparrow \lambda}{ \prod_{j=1}^q b_j \uparrow \lambda}\cdot \frac{\schur_\lambda(M)}{h_\lambda}
\end{eqnarray}
where $\schur_\lambda(M)$ is the Schur polynomial evaluated at the eigenvalues of $M$. This generalizes the classical hypergeometric functions of a complex variable $z$ to (the multiset of eigenvalues of) a square matrix variable $M$, via the following substitutions:
\begin{itemize}
\item the sum over integers is replaced by a sum over partitions;
\item the generalized Pochhammer symbol replaces the rising factorial;
\item the extra factorial that is always introduced for classical hypergeometric functions (by convention then) is replaced by a hook number;
\item powers of $z$ are replaced by Schur functions of the eigenvalues of $M$.
\item They admit integral representations, closely related to Selberg integrals (see \cite{ka1993}).
\end{itemize}

By forming an exponential generating series of Equation~(\ref{finiteNformula}) and with the help of Equation~(\ref{hookcontentFormula}), we are able to obtain the (confluent) hypergeometric function
\begin{eqnarray}
\sum_{r \ge 0} \frac{(\mathcal{M})_N(2k,r)}{(\mathcal{M})_N(2k,0)}\frac{(\complexi z)^r}{r!} &=&  
  \sum_{\mu}\frac{1}{h_\mu^2} \frac{(N\uparrow \mu) ((-k) \uparrow \mu)  }{(-2k) \uparrow \mu} z^{|\lambda|}=\,_1F_1(-k;-2k;z \id_{N \times N}).
\end{eqnarray}
We can then substitute for the RHS of this last equation many different expressions: the theory of hypergeometric functions of matrix arguments also involves integral expressions, differential equations and recurrence relations. However, since we are interested in asymptotics of these expressions for large $N$, for which little theory is developed, this is unfortunately of no real use at the moment. Note though that the hypergeometric function that appears is special, as it is only evaluated at scalar matrices. In that special case, $N\times N$ determinantal formulas have been developed \cite{scalar} (but not in the confluent case).

\section{Further Work}
\label{furtherSec}
The occurrence of the Plancherel measure is not a coincidence. In fact, much can be derived from this, and this will be the basis of further work: we will prove in a subsequent paper that the leading term of $\tilde{X}_{2h}$ has coefficient $\frac{2h!}{h!2^{3h}}$ and is of degree $2h$ lower  than $X_{2h}$. 

\acknowledgements
\label{sec:ack}
The author wishes to acknowledge helpful discussions with Daniel Bump, Richard Hall and Chris Hughes, as well as Alexei Borodin, whose advice was crucial at some stages. 
\bibliographystyle{alpha}
% use the following instead if you encounter problems 
%\bibliographystyle{alpha}
\bibliography{fpsac.bib}
\label{sec:biblio}

\end{document}